\documentclass{amsart}

\usepackage{amssymb}
\usepackage{amsthm}
\usepackage{latexsym,amsmath}

\begin{document}




\title[Almost Ricci-like solitons with torse-forming vertical potential \ldots]
{Almost Ricci-like solitons with torse-forming vertical potential 
of constant length
on almost contact B-metric manifolds}


\author{Mancho Manev}

\address{Department of Algebra and Geometry, University of Plovdiv Paisii Hilendarski,
							24 Tzar Asen St., Plovdiv, 4000, Bulgaria 
\&
Department of Medical Informatics, Biostatistics and E-Learning,
					 Medical University of Plovdiv, 15A Vasil Aprilov Blvd., Plovdiv, 4002, Bulgaria
}
\email{mmanev@uni-plovdiv.bg}

\begin{abstract}
A generalization of Ricci-like solitons with torse-forming potential,
which is constant multiple of the Reeb vector field, is studied.
The conditions under which these solitons are equivalent to almost Einstein-like metrics are given.
Some results are obtained for a parallel symmetric second-order covariant tensor.
Finally, an explicit example of an arbitrary dimension
is given and some of the results are illustrated.
\end{abstract}

%

\keywords{Almost Ricci-like soliton,
almost $\eta$-Ricci soliton,
almost Einstein-like manifold,
almost $\eta$-Einstein manifold,
almost contact B-metric manifold,
torse-forming vector field}



\subjclass[2010]{
53C25, 
53C44, 
70G45} 


\newcommand{\hll}[1]{\colorbox{yellow}{$\displaystyle #1$}}
\newcommand{\ie}{i.\hspace{.5pt}e.\ }
\newcommand{\f}{\varphi}
\newcommand{\g}{\tilde{g}}
\newcommand{\n}{\nabla}
\newcommand{\nn}{\tilde{\n}}
\newcommand{\M}{(M,\A\f,\A\xi,\A\eta,\A{}g)}
\newcommand{\G}{\mathcal{G}}
\newcommand{\I}{\iota}
\newcommand{\W}{\mathcal{W}}
\newcommand{\R}{\mathbb R}
\newcommand{\C}{\mathbb C}
\newcommand{\X}{\mathfrak X}
\newcommand{\F}{\mathcal{F}}
\newcommand{\LL}{\mathcal{L}}
\newcommand{\ta}{\theta}
\newcommand{\om}{\omega}
\newcommand{\lm}{\lambda}
\newcommand{\gm}{\gamma}
\newcommand{\al}{\alpha}
\newcommand{\bt}{\beta}
\newcommand{\sx}{\mathop{\mathfrak{S}}\limits_{x,y,z}}
\newcommand{\D}{\mathrm{d}\hspace{-0.1pt}}
\newcommand{\dsfrac}{\displaystyle\frac}
\newcommand{\const}{\mathrm{const}}

\newcommand{\Div}{\mathrm{div}} 
\newcommand{\grad}{\mathrm{grad}} 
\newcommand{\Hess}{\mathrm{Hess}} 
\newcommand{\tr}{\mathrm{tr}} 
\newcommand{\Span}{\mathrm{span}} 

\newcommand{\abs}[1]{\vert #1 \vert}
\newcommand{\norm}[1]{\left\Vert#1\right\Vert}
\newcommand{\normq}[1]{\lVert#1\rVert ^2}

\newcommand{\A}{\allowbreak{}}

\newtheorem{theorem}{Theorem}[section]
\newtheorem{lemma}[theorem]{Lemma}
\newtheorem{proposition}[theorem]{Proposition}
\newtheorem{corollary}[theorem]{Corollary}

\theoremstyle{definition}
\newtheorem{definition}[theorem]{Definition}
\newtheorem{example}[theorem]{Example}
\newtheorem{xca}[theorem]{Exercise}

\theoremstyle{remark}
\newtheorem{remark}[theorem]{Remark}

\numberwithin{equation}{section}

\newcommand{\thmref}[1]{Theorem~\ref{#1}}
\newcommand{\lemref}[1]{Lemma~\ref{#1}}
\newcommand{\cororef}[1]{Corollary~\ref{#1}}
\newcommand{\propref}[1]{Proposition~\ref{#1}}
\newcommand{\remref}[1]{Remark~\ref{#1}}
\newcommand{\defref}[1]{Definition~\ref{#1}}

\maketitle




\section{Introduction}

In 1982, R.\,S. Hamilton \cite{Ham82} introduced the concept of Ricci flow for Riemannian manifolds.
Ricci solitons are introduced as self-similar solutions of Ricci flows.
Ricci solitons are also natural extensions of Einstein metrics.
Although Ricci solitons were first studied in Riemannian geometry, they and their generalizations have recently been intensively studied for pseudo-Riemannian metrics, mainly on Lorentzian and paracontact metric manifolds
(e.g. \cite{Bla15}, \cite{BlaCra}, \cite{BlaOzg}, \cite{BlaPerAceErd}, \cite{Cra09}, \cite{Mat}, \cite{Ond10}, \cite{YadOzt}, \cite{YadChaSut}).

Torse-forming vector fields are defined by a certain recurrent condition for their covariant derivative with respect to the Levi-Civita connection of the basic metric.
These vector fields are first defined and studied by K. Yano \cite{Yano44}. They are then studied by various
authors (e.g.\ \cite{Chen17}, \cite{MihMih13}, \cite{MihRosVer}). Their research has recently been expanded to include mainly $\eta$-Ricci solitons with torse-forming potential that is orthogonal to $\ker(\eta)$ and the solitons are compatible with various additional tensor structures
(e.g.\ \cite{BlaCra}, \cite{BlaOzg}, \cite{BlaPerAceErd}, \cite{YadOzt}, \cite{YadChaSut}).

In \cite{Man62}, the author begins the study of a generalization of Ricci solitons and $\eta$-Ricci solitons on almost contact B-metric manifolds called Ricci-like solitons, whose potential is the Reeb vector field $\xi$ in two cases: when $\xi$ is torse-forming or the manifold is Sasaki-like.
In \cite{Man63}, he continues to study Sasaki-like manifolds admitting Ricci-like solitons with potential that is pointwise collinear with the Reeb vector field.

The only basic class of almost contact B-metric manifolds that allows a torse-forming Reeb vector field is $\F_5$
\cite{Man62} denoted according to the Ganchev-Mi\-ho\-va-Gribachev classification \cite{GaMiGr}. This class is considered in most of the present work and $\F_5$ is the counterpart of the class of $\bt$-Kenmotsu manifolds among the classes of almost contact metric manifolds.
Ricci solitons and $\eta$-Ricci solitons on manifolds with almost contact or almost paracontact structures, both of Kenmotsu type, have been studied in a number of works
(e.g.\ \cite{AyaYil}, \cite{BagIngAsh13}, \cite{Bla15}, \cite{CalCra}, \cite{YadChaSut}).

%


It is known from \cite{Levy} the theorem of H. Levy, which states that a second-order symmetric parallel non-singular tensor on a space of constant curvature is a constant multiple of the metric tensor.
In \cite{Sha89}, R. Sharma gives a generalization of Levy's theorem
in a non-flat real space form of dimension greater than two.
Moreover, he proved that the only symmetric (resp., anti-symmetric) tensor of this type 
in a non-flat complex space form is the K\"ahlerian metric (resp., 2-form) up to a constant multiple.
Similar investigations on other type of manifolds are made in
\cite{BejCra14}, \cite{BlaPerAceErd}, \cite{Das07}, \cite{De96}, \cite{Li97}.

In the present paper, our aim is to generalize and develop our study of Ricci-like solitons on almost contact B-metric manifolds by studying a potential that is torse-forming with constant length and vertical direction---in the sense that
it is orthogonal to the contact (horizontal) distribution $\ker(\eta)$ with respect to the basic metric. 
The paper is organized as follows.
After the present introduction, in Section~2, we recall basic definitions and facts for almost contact B-metric manifolds and torse-forming vector fields, and then we find immediate properties of these manifolds that have a torse-forming vertical vector field with constant length.
In Section~3, we prove necessary and sufficient conditions for the studied manifolds to admit almost Ricci-like solitons and to be almost Einstein-like manifolds.
In Section~4, we give a characterization for almost Ricci-like solitons on the studied manifold
concerning a parallel symmetric $(0,2)$-tensor.
In Section~5, we construct an explicit example of a smooth manifold of dimension $(2n+1)$ equipped with an almost contact B-metric structure. Then, we show that the manifold of the considered type
and
the results for this example support the relevant assertions in the previous sections.

\section{Torse-forming constant-length vertical vector fields on almost contact B-metric manifolds}

The subject of our study are \emph{almost contact B-metric manifolds}. A differentiable manifold $M$ of this type has dimension $(2n+1)$ and it is denoted by $\M$, where  $(\f,\xi,\eta)$ is an almost
contact structure and $g$ is a B-metric. More precisely, $\f$ is an endomorphism
on $\Gamma(TM)$, $\xi$ is a Reeb vector field, $\eta$ is its dual contact 1-form and
$g$ is a pseu\-do-Rie\-mannian
metric  of signature $(n+1,n)$ satisfying the following conditions \cite{GaMiGr}
\begin{equation}\label{strM}
\begin{array}{c}
\f\xi = 0,\qquad \f^2 = -\I + \eta \otimes \xi,\qquad
\eta\circ\f=0,\qquad \eta(\xi)=1,\\[4pt]
g(\f x, \f y) = - g(x,y) + \eta(x)\eta(y),
\end{array}
\end{equation}
where $\I
$ is the identity transformation on $\Gamma(TM)$.

In the latter equality and further, $x$, $y$, $z$, $w$ will stand for arbitrary elements of $\Gamma(TM)$ or vectors in the tangent space $T_pM$ of $M$ at an arbitrary
point $p$ in $M$.

Some immediate consequences of \eqref{strM} are the following equations
\begin{equation}\label{strM2}
\begin{array}{ll}
g(\f x, y) = g(x,\f y),\qquad &g(x, \xi) = \eta(x),
\\[4pt]
g(\xi, \xi) = 1,\qquad &\eta(\n_x \xi) = 0,
\end{array}
\end{equation}
where $\n$ is the Levi-Civita connection of $g$.

An important characteristic of these manifolds is the existence of one more B-metric together with $g$.
The associated metric $\g$ of $g$ on $(M,\f,\xi,\eta)$ is defined by
\[ 
\g(x,y)=g(x,\f y)+\eta(x)\eta(y)
\]
and it is easy to verify that $\g$ is also a B-metric on $(M,\f,\xi,\eta)$.

A classification of almost contact B-metric manifolds, consisting of eleven basic classes $\F_i$, $i\in\{1,2,\dots,11\}$, is given in
\cite{GaMiGr}. This classification is made with respect
to the (0,3)-tensor $F$ defined by
\[ 
F(x,y,z)=g\bigl( \left( \nabla_x \f \right)y,z\bigr).
\]
The tensor $F$ possess the following basic properties:
\[ 
\begin{array}{l}
F(x,y,z)=F(x,z,y)
=F(x,\f y,\f z)+\eta(y)F(x,\xi,z)
+\eta(z)F(x,y,\xi),\\[4pt]
F(x,\f y, \xi)=(\n_x\eta)y=g(\n_x\xi,y).
\end{array}
\]

The intersection of the basic classes is the special class $\F_0$,
determined by the condition $F=0$, and it is known as the
class of the \emph{cosymplectic B-metric manifolds}.

The Lee forms of $\M$ are the following 1-forms
associated with $F$:
\[ 
\theta(z)=g^{ij}F(e_i,e_j,z),\quad
\theta^*(z)=g^{ij}F(e_i,\f e_j,z), \quad \omega(z)=F(\xi,\xi,z),
\]
where $\left(g^{ij}\right)$ is the inverse matrix of the
matrix $\left(g_{ij}\right)$ of $g$ with respect to a basis $\left\{e_i;\xi\right\}$ $(i=1,2,\dots,2n)$ of
$T_pM$.  Obviously,
$\om(\xi)=\ta^*\circ\f+\ta\circ\f^2=0$ are valid.

In \cite{Man62}, it is given the following definition.
An almost contact B-metric manifold $\M$ is said to be
\emph{Einstein-like} if its Ricci tensor $\rho$ satisfies the condition
\begin{equation}\label{defEl}
\begin{array}{l}
\rho=a\,g +b\,\g +c\,\eta\otimes\eta
\end{array}
\end{equation}
for some triplet of constants $(a,b,c)$.

In particular, when $b=0$ and $b=c=0$, the manifold is called an \emph{$\eta$-Einstein manifold} and an \emph{Einstein manifold}, respectively.
If $a$, $b$, $c$ are functions on $M$, then the manifold satisfying condition \eqref{defEl} is called \emph{almost Einstein-like}, \emph{almost $\eta$-Einstein} and \emph{almost Einstein}, respectively.

As a consequence of \eqref{defEl} we obtain that the corresponding scalar curvature $\tau$ and its associated quantity $\tau^*$  on an almost Einstein-like manifold have the form:
\begin{equation}
\label{tau-El}
\tau=(2n+1)a +b +c,\\[4pt]
\end{equation}
\begin{equation*}
\label{tau*-El}
\tau^*=-2n\,b,
\end{equation*}
where $\tau^*$ is defined by $\tau^*:=g^{ij}\rho(e_i,\f e_j)$ in an arbitrary basis $\{e_i\}$, $i\in\{1,\dots,2n+1\}$ of
$T_pM$.


A vector field $\vartheta$ on a (pseudo-)Riemannian manifold $(M,g)$ is called \emph{torse-form\-ing vector field} if it
satisfies the following condition for arbitrary vector field $x\in \Gamma(TM)$ 
\begin{equation}\label{tf-v}
	\n_x \vartheta = f\,x + \gm(x)\vartheta,
\end{equation}
where $f$ is a differentiable function and $\gm$ is a 1-form \cite{Yano44}, \cite{Sch54}.
The 1-form $\gm$ is called the \emph{generating form} and
the function $f$ is called the \emph{conformal scalar} of $\vartheta$ \cite{MihMih13}.

In \cite{Man62}, the torse-forming Reeb vector field $\xi$ is studied. Then, according to \eqref{strM2}, the 1-form $\gm$ on an almost contact B-metric manifold $\M$ is  determined by $\gm=-f\,\eta$ and then we have 
\begin{equation}\label{tf}
\begin{array}{l}
		\n_x \xi=-f\,\f^2x,\qquad \left(\n_x \eta \right)(y)=-f\,g(\f x,\f y).
		\end{array}
\end{equation}

As a corollary of \eqref{tf}, it is shown that $\ta^*(\xi)=2n\,f$ and $\ta(\xi)=\om=0$. 
Then, by the components of $F$ in the basic classes $\F_i$ ($i\in\{1,2,\dots,11\}$), given in \cite{HM1}, it is deduced in \cite{Man62} that the class of
the manifolds $\M$ with torse-forming $\xi$ is
$\F_1\oplus\F_2\oplus\F_3\oplus\F_5\oplus\F_6\oplus\F_{10}$.
Among the basic classes, only $\F_5$ can contain such manifolds.

If the Reeb vector field $\xi$ on $\M\in\F_5$ is torse-forming with a conformal scalar $f$, then it is valid the following
\begin{equation}\label{tfF5}
		\left(\n_x \f \right)y=-f\{g(x,\f y)\xi+\eta(y)\f x\}.
\end{equation}
Then, for an Einstein-like manifold $\M$ with torse-forming $\xi$, 
it is known that
 the condition for Ricci-symmetry is equivalent to the Einstein condition \cite{Man62}.

\begin{remark}\label{rem:types}
Some special types of torse-forming vector fields have been considered in various studies.
A vector field $\vartheta$ determined by \eqref{tf-v} is called respectively:
\begin{description}
	\item - \emph{torqued}, if $\gm(\vartheta) = 0$;
	\item - \emph{concircular}, if $\gm = 0$;
	\item - \emph{concurrent}, if $f - 1=\gm = 0$;
	\item - \emph{recurrent}, if $f = 0$;
	\item - \emph{parallel}, if $f = \gm = 0$.
\end{description}
\end{remark}
We further exclude from our consideration the trivial case when $f = 0$, since it implies the parallelism of $\vartheta$.

Further, we consider a torse-forming vector field $\vartheta$ on $\M$, i.e.\ \eqref{tf-v} is valid.
Moreover, we suppose that $\vartheta$ is a constant multiple of $\xi$, i.e. $\vartheta=k\,\xi$,
where $k$ is a nonzero constant on $M$ and obviously
$k=\eta(\vartheta)$ holds true. Therefore, $\vartheta$ belongs to the vertical distribution $H^\bot=\Span(\xi)$, which is orthogonal to the contact distribution $H=\ker(\eta)$ with respect to $g$.
For these reasons, we call such a vector field $\vartheta$ \emph{constant-length vertical}.

\begin{theorem}\label{thm:curv-tfv}
Let a vector field $\vartheta$ on $\M$ be torse-forming with a conformal scalar $f$ and generating form $\gm$. Moreover, let $\vartheta$ be constant-length vertical with a constant $k$ of proportionality to $\xi$. Then we have:
\begin{enumerate}
	\item[$(i)$] 	$\xi$ is a geodesic vector field and $\eta$ is a closed 1-form;
	
	\item[$(ii)$]   $\xi$ is a torse-forming vector field with conformal scalar ${f}/{k}$
and generating form
$\gm=-({f}/{k}) \eta$;

	\item[$(iii)$]  the following equalities for the curvature tensor $R$, the Ricci tensor $\rho$ and the sectional curvature $K$ are valid:
\end{enumerate}
%
\begin{gather}
\label{tfv-Rxyxi}
R(x,y)\xi=- \frac{1}{k^2}\left\{\left[k\,\D{f}(x)+f^2\eta(x)\right] \f^2 y
-\left[k\,\D{f}(y)+f^2\eta(y)\right] \f^2 x\right\},
\\[4pt]
\label{tfv-Rxxixi}
R(x,\xi)\xi=\frac{1}{k^2}\left\{k\,\D{f}(\xi)+f^2\right\} \f^2 x,
\\[4pt]
\label{tfv-royxi}
\rho(y,\xi)=- \frac{1}{k^2}\left\{(2n-1)k\,\D{f}(y)+\left[k\,\D{f}(\xi)+2n\,f^2\right]\eta(y)\right\},
\\[4pt]
\label{tfv-roxixi}
\rho(\xi,\xi)=- \frac{2n}{k^2}\left\{k\,\D{f}(\xi)+f^2\right\},
\\[4pt]
\label{tfv-kxix}
K(\xi,x)=- \frac{1}{k^2}\left\{k\,\D{f}(\xi)+f^2\right\},
\end{gather}
\ie the sectional curvature of any $\xi$-section   $(\xi,x)$, $x\notin H^\bot$,
does not depend on $x$. 
\end{theorem}

\begin{proof}
By virtue of \eqref{tf-v} and  $\vartheta=k\,\xi$, we obtain
\begin{equation}\label{dkxxi}
    k\n_x \xi = fx+k\gm(x)\xi,
\end{equation}
which after applying $\eta$ gives
\begin{equation}\label{al-tf-v}
\gm=-\frac{f}{k} \eta,
\end{equation}
Combining \eqref{dkxxi} and \eqref{al-tf-v}, we get
\begin{equation}\label{nxxi}
\n_x \xi =- \frac{f}{k} \f^2 x.
\end{equation}

Equality \eqref{nxxi}  shows that $\xi$ is a geodesic vector field and $\eta$ is a closed 1-form, \ie $(i)$ is true. Moreover, rewriting \eqref{nxxi} in the form
$\n_x \xi =(f/k) \{x- \eta(x)\xi\}$ and comparing with \eqref{tf-v}, we find that $\xi$ is torse-forming with a conformal scalar $f/k$ and the same generating form $\gm$ as for $\vartheta$ bearing in mind \eqref{al-tf-v}, 
i.e. $(ii)$ is also valid.



%


Finally, using  \eqref{nxxi}, we calculate the  curvature equalities given in
$(iii)$. 
\end{proof}

By virtue of \eqref{nxxi} and \eqref{tfF5}, we deduce that for the torse-forming vector field $\vartheta$ on $\M\in\F_5$ with a conformal scalar $f/k$ it is valid the following
\begin{equation}\label{tfvF5}
\begin{array}{l}
		\left(\n_x \f \right)y=-\frac{f}{k}\{g(x,\f y)\xi+\eta(y)\f x\}.
\end{array}
\end{equation}

Bearing in mind \eqref{tfv-Rxxixi}, we note that if the conformal scalar $f$ satisfies the equation $k\,\D f(\xi) +f^2 = 0$, then $R(x,\xi)\xi$
degenerates for arbitrary $x$.
Therefore, we give the following

\begin{definition}\label{def:reg}
Torse-forming constant-length vertical vector field $\vartheta$ on an almost contact B-met\-ric manifold $\M$
is called \emph{regular} if its generating function $f$ and 
its constant $k$ of proportionality to $\xi$ satisfy the condition
$k\,\D f(\xi) +f^2 \neq 0$.
\end{definition}

\begin{remark}\label{rem:reg=f}
Let us give an example of a regular (respectively, non-regular) torse-forming constant-length vertical vector field $\vartheta$ on $\M$.
Suppose that $f$ is a function of $t\in\R\setminus \{0\}$ determined by $f={(1+\sqrt{3})k}/({2t})$. 
Then we check that
$\vartheta$ with this conformal scalar $f$ is regular because we get 
$
kf'+f^2=k^2/(2t^2)>0.
$
Otherwise, let $f$ be determined by $f=k/t$, $t\neq 0$. Then, we find that  $kf'+f^2=0$ and therefore
$\vartheta$ with the latter conformal scalar $f$ is non-regular.
\end{remark}

\begin{remark}
Using  \eqref{al-tf-v}, we refine the definition condition \eqref{tf-v}
for the torse-forming constant-length vertical vector field $\vartheta$ in the form
\begin{equation}\label{tf-v=}
	\n_x \vartheta = -f \f^2x,
\end{equation}
which means that $\vartheta$ is concircular on $H$ and it is parallel along $H^\bot$.
Furthermore, obviously in the present case $\vartheta$ cannot be any of the other special types given in \remref{rem:types}.
\end{remark}


\section{Ricci-like solitons with torse-forming constant-length vertical potential on almost contact B-metric manifolds}


According to  \cite{Man63}, an almost contact B-metric manifold $\M$ is said to admit a
\emph{Ricci-like soliton} with potential vector field $\vartheta$ if its Ricci tensor $\rho$ satisfies the following condition for a triplet of constants $(\lm,\mu,\nu)$
\begin{equation}\label{defRl-v}
\begin{array}{l}
\frac12 \mathcal{L}_{\vartheta} g  + \rho + \lm\, g  + \mu\, \g  + \nu\, \eta\otimes \eta =0,
\end{array}
\end{equation}
where $\LL$ denotes the Lie derivative.

If $\mu=0$ (respectively, $\mu=\nu=0$), then \eqref{defRl-v} defines an \emph{$\eta$-Ricci soliton}
(respectively, a \emph{Ricci soliton}) on $\M$.

If $\lm$, $\mu$, $\nu$ are functions on $M$, then the soliton is called \emph{almost Ricci-like soliton}, \emph{almost $\eta$-Ricci soliton} and \emph{almost Ricci soliton}, respectively \cite{PiRiRiSe11}.

A Ricci soliton is called \emph{shrinking}, \emph{steady} or \emph{expanding}
depending on whether $\lm$ is negative, zero or positive, respectively \cite{ChoLuNi}.

For the case of a Ricci-like soliton with potential $\vartheta=\xi$, let us recall the following
\begin{theorem}[\cite{Man62}]\label{thm:Rltf}
Let $\xi$ on $\M$, $\dim M=2n+1$,  be torse-forming with a conformal scalar $f$.
The manifold admits
a Ricci-like soliton with potential $\xi$ and constants $(\lm,\mu,\nu)$
if and only if
the manifold is Einstein-like with constants $(a,b,c)$
provided that
$f$ is a constant and the following equalities are satisfied:
\begin{equation}\label{tfElRl-const}
a+\lm+f=0,\qquad b+\mu=0,\qquad c+\nu-f=0.
\end{equation}

In particular, we have:
\begin{enumerate}
	\item[$(i)$]    The manifold admits an $\eta$-Ricci soliton with potential $\xi$ and constants $(\lm,0,\nu)$ if and only if
the manifold is $\eta$-Einstein with constants $(a,0,c)$, where the equality  $a+c+\lm+\nu=0$ is valid.

	\item[$(ii)$]   The manifold admits a Ricci soliton with potential $\xi$ and constant $\lm$ if and only if
the manifold is $\eta$-Einstein with constants $(-\lm-f,0,f)$.

	\item[$(iii)$]   The manifold is Einstein with constant $a$ if and only if
it admits an $\eta$-Ricci with potential $\xi$ and constants $(-a-f,0,f)$.
\end{enumerate}
\end{theorem}

Now, we generalize the latter result for almost Ricci-like solitons with
torse-forming constant-length vertical potential.

\begin{theorem}\label{thm:tf_k=const}
Let $\M$ be $(2n+1)$-dimensional 
and $\vartheta$ be a vector field on $M$, which is vertical with a constant $k=\eta(\vartheta)$ as well as
torse-forming with a function $f$ as a conformal scalar.
Moreover, let $a$, $b$, $c$, $\lm$, $\mu$, $\nu$ be functions on $M$ that satisfy the following equalities:
\begin{gather}
\label{tfElRl-0}
a+\lm+f=0,\qquad b+\mu=0,\qquad c+\nu-f=0,
\\[4pt]
\label{Dfxi-tfRlaEl}
\lm+\mu+\nu=-a-b-c=\frac{2n}{k^2}\left\{k\,\D{f}(\xi)+f^2\right\}.
\end{gather}

Then, $M$ admits an almost Ricci-like soliton with potential $\vartheta$ and functions $(\lm,\mu,\A\nu)$
if and only if
$M$ is almost Einstein-like with functions $(a,b,c)$.
\end{theorem}

\begin{proof}
For a torse-forming constant-length vertical vector field $\vartheta$ on $\M$, we have the formula in \eqref{tf-v=}.
It implies the following equality
\begin{equation}\label{Lvg}
\left(\LL_\vartheta g\right)(x,y)=-2 f g(\f x, \f y).
\end{equation}
Then, the condition for an almost Ricci-like soliton with potential $\vartheta$ given in \eqref{defRl-v} takes the form
\begin{equation}\label{ro-v}
\begin{array}{l}
\rho(x,y)=(\lm+f)g(\f x, \f y)-\mu g(x, \f y)-(\lm+\mu+\nu)\eta(x)\eta(y),
\end{array}
\end{equation}
which coincides with the condition for an almost Einstein-like manifold given in
\eqref{defEl}, assuming the conditions in \eqref{tfElRl-0}.

As consequences of  \eqref{defEl} and \eqref{ro-v} we have the following expressions, respectively:
\begin{equation}\label{roxixi-El}
\rho(\xi,\xi)=a+b+c
\end{equation}
and  
\begin{gather}
\label{roxxi-Rl}
\rho(x,\xi)=-(\lm+\mu+\nu)\eta(x),
\\[4pt]
\label{roxixi-Rl}
\rho(\xi,\xi)=-(\lm+\mu+\nu).
\end{gather}
Comparing \eqref{tfv-roxixi}, \eqref{roxixi-El} and \eqref{roxixi-Rl}, we get the dependencies in \eqref{Dfxi-tfRlaEl} between the involved functions.
\end{proof}

\begin{remark}\label{rem:reg}
Taking into account \defref{def:reg} and \eqref{Dfxi-tfRlaEl},
we conclude that the condition for regularity of $\vartheta$ is equivalent to the non-vanishing of any of the sums $a+b+c$ and $\lm+\mu+\nu$.
\end{remark}

\begin{corollary}\label{cor:3.3}
Under the hypothesis of 
\thmref{thm:tf_k=const} we have in particular:
\begin{enumerate}
	\item[$(i)$]    The manifold admits an almost $\eta$-Ricci soliton with potential $\vartheta$ if and only if
it is almost $\eta$-Einstein.

	\item[$(ii)$]    The manifold admits an almost Ricci soliton with potential $\vartheta$ and function $\lm$ if and only if
it is almost $\eta$-Einstein (which is not almost Einstein) with functions $(a,b,c)=(-\lm-f,0,f)$.

	\item[$(iii)$]   The manifold is almost Einstein with function $a$ if and only if
it admits an almost $\eta$-Ricci soliton (which is not almost Ricci soliton) with potential $\vartheta$ and functions $(\lm,\mu,\nu)=(-a-f,0,f)$.
\end{enumerate}
\end{corollary}

\begin{corollary}\label{cor:3.5}
Under the hypothesis of 
\thmref{thm:tf_k=const} we have that
$M$ is almost Einstein-like with functions $(a,b,c)$ and scalar curvature $\tau$. Then the conformal scalar $f$ satisfies the following equation
\begin{equation}\label{ddtf}
kf'+f^2
=k^2\left(a-\frac{\tau}{2n}\right),
\end{equation}
where $f'$ is
the derivative of the function $f=f(t)$ and $t$ is a coordinate on $H^\bot$ and
the sectional curvature of an arbitrary $\xi$-section is
\begin{equation}\label{Kxi}
K_\xi=\frac{\tau}{2n}-a.
\end{equation}
\end{corollary}

\begin{proof}
Comparing \eqref{tfv-royxi} and \eqref{roxxi-Rl}, we obtain the following equality
\[
(2n-1)k\,\D{f}(x)=\left\{k^2(\lm+\mu+\nu)-\left[k\,\D{f}(\xi)+2n f^2\right]\right\}\eta(x),
\]
which implies that the function $f$ is constant on $H$ and its derivative on $H^\bot$ is determined by
\begin{equation}\label{Dfxi-Rl}
\D{f}(\xi)=-\frac{1}{k}f^2+\frac{k}{2n}(\lm+\mu+\nu)
\end{equation}
and then the following formula holds true
\begin{equation*}\label{Dfx-Rl}
\D{f}(x)=-\left[\frac{1}{k}f^2-\frac{k}{2n}(\lm+\mu+\nu)\right]\eta(x),
\end{equation*}
\ie the gradient of $f$ is vertical, $\grad{f}\in H^\bot$.
Therefore, due to \eqref{tau-El}, \eqref{Dfxi-tfRlaEl} and \eqref{Dfxi-Rl}, we obtain the expression in \eqref{ddtf}.
The equality in \eqref{Kxi} follows from \eqref{tfv-kxix} and \eqref{ddtf}.
\end{proof}


\section{Parallel symmetric second-order tensor on almost contact B-metric manifold with torse-forming
constant-length vertical vector field}

On an almost contact B-metric manifold $\M$ with torse-forming con\-stant-length vertical vector field $\vartheta$,
we use the curvature equalities \eqref{tfv-Rxyxi}, \eqref{tfv-Rxxixi}, \eqref{tfv-royxi} and \eqref{tfv-roxixi}.



\begin{theorem}\label{thm:h-tf}
On an almost contact B-metric manifold with regular torse-forming constant-length vertical vector field $\vartheta$, every symmetric
second-order covariant tensor that is parallel with respect to the Levi-Civita connection of the B-metric,
is a constant multiple of this metric.
\end{theorem}

\begin{proof}
Let $h$ be a symmetric
$(0,2)$-tensor field which is parallel, i.e. $\n h=0$ with respect to the Levi-Civita connection of $g$.
Then, because of \eqref{tfv-Rxyxi} and the Ricci identity for $h$, the following property is valid
\begin{equation}\label{hxi}
	h\left(R(x,y)\xi,\,\xi\right)=0.
\end{equation}

Using \eqref{tfv-Rxyxi} and \eqref{hxi}, 
we obtain
\[ 
\begin{array}{l}
	 \left[k\,\D{f}(x)+f^2\eta(x)\right] h(\f^2 y,\xi)
-\left[k\,\D{f}(y)+f^2\eta(y)\right] h(\f^2 x,\xi)=0.
\end{array}
\]
The latter equality, by substitution $y$ for $\xi$ and recalling \eqref{strM}, provides the following
\[ 
\begin{array}{l}
		\{ k\,\D f(\xi) +f^2\}\{ h(x,\xi) - h(\xi,\xi)\eta(x)\} = 0.
\end{array}
\]

Since $\vartheta$ is regular, then the latter equality gives the following property of $h$
\begin{equation}\label{tf-h}
h(x,\xi) = h(\xi,\xi)\eta(x).
\end{equation}
Then,
the claim that $h(\xi,\xi)$ is a constant follows from the latter property, the parallelism of $h$ and \eqref{nxxi}.

The covariant derivative of \eqref{tf-h} with respect to $y$ implies the following equality
due to \eqref{strM} and \eqref{nxxi} 
\[ 
h(x,y) = h(\xi,\xi)g(x,y).
\]
Therefore, the statement is true.
\end{proof}

Now, we apply \thmref{thm:h-tf} to Ricci-like solitons and we deduce the following

\begin{theorem}\label{thm:h-tf_R-l}
Let $\M$ be a $(2n+1)$-dimensional almost contact B-metric manifold with regular torse-forming constant-length vertical vector field $\vartheta$ for a conformal scalar $f$ and a constant $k$ of proportionality to $\xi$. Additionally, let $h$ be determined by the following way
$h:=\frac12 \LL_\vartheta g+\rho + \mu\, \g  + \nu\, \eta\otimes \eta$, where $\mu$ and $\nu$ are differentiable functions on $M$.
Then the following conditions are equivalent:
\begin{enumerate}
	\item[$(i)$]    $h$ is parallel with respect to $\n$ of $g$;

	\item[$(ii)$]   $M$ admits an almost Ricci-like soliton with potential $\vartheta$ and functions $(\lm,\mu,\A\nu)$, where $\lm=\mathrm{const}$, so that the following condition is satisfied
\end{enumerate}
\begin{equation}\label{tf-Rl-lmn}
\lm+\mu+\nu=\frac{2n}{k^2}\left\{k\,\D{f}(\xi)+f^2\right\}.
\end{equation}
\begin{enumerate}
	\item[$(iii)$]   $M$ is almost Einstein-like with functions $(a,b,c)$, so the following condition is met
\[ 
a+b+c=-\frac{2n}{k^2}\left\{k\,\D{f}(\xi)+f^2\right\}.
\]
\end{enumerate}
\end{theorem}

\begin{proof}
When $\vartheta$ is torse-forming, we have \eqref{Lvg} and hence the following expression
\[ 
h=\rho + f\,g +\mu\, \g  + (\nu-f) \eta\otimes \eta,
\]
which due to \eqref{tfv-roxixi} gives
\begin{equation}\label{h-tf-Rl-xixi}
	h(\xi,\xi)=
	- \frac{2n}{k^2}\left\{k\,\D{f}(\xi)+f^2\right\} + \mu + \nu.
\end{equation}

The condition that $M$ admits an almost Ricci-like soliton with potential $\vartheta$ is equivalent to $h=-\lm\,g$,
taking into account \eqref{defRl-v}. The latter expression of $h$ by $g$ implies $h(\xi,\xi)=-\lm$, which together with \eqref{h-tf-Rl-xixi} yields
\eqref{tf-Rl-lmn}. Moreover, $\lm=-h(\xi,\xi)$ means that $\lm$ is a constant, which completes the proof of the equivalence of  $(i)$ and $(ii)$.

Bearing in mind \thmref{thm:tf_k=const}, we deduce that the assertions  $(ii)$ and $(iii)$ are equivalent under the given conditions.
\end{proof}

\begin{corollary}\label{cor:h-tf_RlEl}
Let $\M$ be Einstein-like with regular torse-forming $\xi$ for a conformal scalar $f$ and let $h$ be determined by
$h:=\frac12\left(\LL_\xi g\right)+\rho + \mu\, \g  + \nu\, \eta\otimes \eta$, where $\mu,\nu\in\R$.
Then $h$ is parallel with respect to $\n$ of $g$ if and only if $M$ admits a Ricci-like soliton with potential $\xi$ and constants $(\lm,\mu,\nu)$, satisfying \eqref{tfElRl-const} for the constant  $f$ and
\begin{equation}\label{h-tf-RlEl-lmn}
\lm+\mu+\nu=2n\,f^2>0.
\end{equation}
\end{corollary}
\begin{proof}
The statement follows from \thmref{thm:h-tf_R-l} for $k=1$ and constants $(\lm,\mu,\nu)$, $(a,b,c)$
as well as $f=\mathrm{const}$, according to \thmref{thm:Rltf}.
Hence, \eqref{h-tf-RlEl-lmn} is also valid.
\end{proof}

\begin{remark}
\cororef{cor:h-tf_RlEl} can be specialized.
On the one hand, if $b=\mu=0$ holds, then
it follows that the Einstein-like condition is reduced to the $\eta$-Einstein condition
if and only if
the notion of a Ricci-like soliton is reduced to the notion of an $\eta$-Ricci soliton.
On the other hand, if suppose $\mu=\nu=0$, then the second-order tensor $h:=\frac12\left(\LL_\xi g\right)+\rho$ is parallel if and only if
the manifold admits an expanding Ricci soliton with potential $\xi$ and constant $\lm=2n\,f^2$.
\end{remark}

\section{Example}

Bearing in mind Example 1 in \cite{GaMiGr}, let us consider the following cosymplectic B-metric manifold $(M,\f,\A\xi,\A\eta,g)$, where
\[
\begin{array}{c}
M\subset \R^{2n+1}=\left\{\left(x^1,x^2,\dots,x^n;x^{n+1},x^{n+2},\dots,x^{2n};t\right)\right\}, \\[4pt]
x^{n+i}\neq 0,\quad i\in\{1,2,\dots,n\}, \qquad t\neq 0, \\[4pt]
\f\left(\frac{\partial}{\partial x^i}\right)=\left(\frac{\partial}{\partial x^{n+i}}\right),\qquad
\f\left(\frac{\partial}{\partial x^{n+i}}\right)=-\left(\frac{\partial}{\partial x^i}\right),\qquad
\f\left(\frac{\partial}{\partial t}\right)=0,\\[4pt]
\xi=\frac{\partial}{\partial t},\qquad \eta=\D{t}, \\[4pt]
g(z,z)=-\delta_{ij}\left(z^i z^j-z^{n+i} z^{n+j}\right)+\left(z^{2n+1}\right)^2,\\[4pt]
\g(z,z)=\delta_{ij}\left(z^i z^{n+j} + z^{n+i} z^j\right)+\left(z^{2n+1}\right)^2
\end{array}
\]
for $z=z^i\frac{\partial}{\partial x^i}+z^{n+i}\frac{\partial}{\partial x^{n+i}}+z^{2n+1}\frac{\partial}{\partial t}$,
$i\in\{1,2,\dots,n\}$ and $\delta_{ij}$ are the Kronecker's symbols.

Next, we apply a contact conformal transformation of $g$
and obtain a B-met\-ric $\bar{g}$ as follows
\[
\bar{g}=e^{2u}\cos{2v}\, g + e^{2u}\sin{2v}\, \g +\left(1-e^{2u}\cos{2v} - e^{2u}\sin{2v}\right)\eta\otimes\eta,
\]
determined by functions
\[
\begin{split}
u&=\frac12\sum_{i=1}^n\left\{\ln\left[\left(x^i\right)^2+\left(x^{n+i}\right)^2\right]\right\}+\ell(t)
,\\[4pt]
v&=\sum_{i=1}^n \arctan\frac{x^i}{x^{n+i}},
\end{split}
\]
where $\ell(t)$ is an arbitrary twice differentiable function on $\R$ such that $\ell'\neq 0$. 
Then, it is proved in \cite{ManGri93} that
$(M,\f,\xi,\eta,\bar{g})$ belongs to the subclass of $\F_5$ defined by the additional condition $\theta^*$ to be closed. 
Moreover, $(M,\f,\xi,\eta,\bar{g})$ does not belong to $\F_0$, since
$\theta^*=2n\D{u}(\xi)\eta$ and $\D{u}(\xi)=\ell'(t)\neq 0$. Therefore, we have
\begin{equation}\label{th*}
\theta^*(\xi)=2n\,\ell'.
\end{equation}
Furthermore, according to Theorem 3.1 and Theorem 5.4 in \cite{ManGri94}, 
we obtain that 
the curvature tensor of $\bar{g}$ has the following form
\begin{equation}\label{R-ex}
\begin{array}{l}
R(x,y,z,w)=-\ell''\{\eta(y)\eta(z)g(x,w)-\eta(x)\eta(z)g(y,w)\\[4pt]
\phantom{R(x,y,z,w)=-\ell''\{}
+\eta(x)\eta(w)g(y,z)-\eta(y)\eta(w)g(x,z)\}\\[4pt]
\phantom{R(x,y,z,w)=}
-(\ell')^2\{g(y,z)g(x,w)-g(x,z)g(y,w)\}.
\end{array}
\end{equation}
 Therefore, the Ricci tensor, the scalar curvature and its associated quantity of $(M,\f,\xi,\eta,\bar{g})$ have the following form
\begin{equation}\label{ro-ex}
\begin{array}{l}
\rho=-\left[\ell''+2n\left(\ell'\right)^2\right]g-(2n-1)\ell''\,\eta\otimes\eta,\\[4pt]
\tau=-4n\,\ell''-2n(2n+1)\left(\ell'\right)^2,\qquad \tau^*=0.
\end{array}
\end{equation}
Hence, the obtained manifold is almost Einstein-like with functions
\begin{equation}\label{abcell}
a=-\left[\ell''+2n\left(\ell'\right)^2\right],\qquad b=0,\qquad c=-(2n-1)\ell''
\end{equation}
or rather it is almost $\eta$-Einstein. 

\begin{remark}\label{rem:ell}
In particular, $(M,\f,\xi,\eta,\bar{g})$ can be almost Ein\-stein, if and only if $\ell$ is a linear function of $t$, which implies that the obtained manifold is just Einstein, or more precisely, a hyperbolic space form. 
\end{remark}

Let us consider a torse-forming constant-length vertical vector field $\vartheta$ with a conformal scalar $f$ and a constant $k$ of proportionality to $\xi$. Then, $\vartheta$ is defined by \eqref{tf-v=}.
According to \eqref{tfvF5} and
\eqref{th*}, we get
\begin{equation}\label{ell}
\ell'=\frac{f}{k}.
\end{equation}

Obviously, \eqref{R-ex}, \eqref{ro-ex} and \eqref{ell} are in accordance with
the equalities in part $(iii)$ of \thmref{thm:curv-tfv}.

Equality \eqref{ell} converts \eqref{abcell} in the following form
\begin{equation}\label{abc-ex}
a=-\frac{1}{k^2}\left[k\,f'+2n\,f^2\right],\qquad b=0,\qquad c=\frac{1-2n}{k}f'
\end{equation}
and the scalar curvature is
\begin{equation}\label{tau-ex}
\tau=-\frac{2n}{k^2}\left[2k\,f'+(2n+1)f^2\right].
\end{equation}
Equalities \eqref{abc-ex} and \eqref{tau-ex} support \cororef{cor:3.5}.

Since \eqref{tfElRl-0} and \eqref{Dfxi-tfRlaEl} are satisfied for \eqref{abc-ex}  and
\begin{equation}\label{lmn-ex}
\lm=\frac{1}{k^2}\left[k\,f'+2n\,f^2\right]-f,\qquad \mu=0,\qquad \nu=\frac{2n-1}{k}f'+f,
\end{equation}
according to \thmref{thm:tf_k=const}, we get that $(M,\f,\xi,\eta,\bar{g})$ admits an almost Ricci-like soliton with potential $\vartheta$ and functions $(\lm,\mu,\nu)$ from \eqref{lmn-ex}.
Furthermore, it supports \cororef{cor:3.3}, case $(i)$.

The case $(ii)$ of \cororef{cor:3.3} can be illustrated by the constructed example $(M,\f,\xi,\eta,\bar{g})$ for $f=q\exp\left(-\frac{kt}{2n-1}\right)$, $q=\const\neq 0$. It is easy to check that the corresponding torse-forming vector field $\vartheta$ is regular. 

In particular case of \remref{rem:ell}, the obtained manifold is an example of case $(iii)$ in \cororef{cor:3.3}, when the functions $a$, $\lm$, $\mu$ and $f$ are constants. Then, the corresponding torse-forming vector field $\vartheta$ is regular because $\ell'\neq 0$ and $f\neq 0$.

Let us consider the torse-forming potential $\vartheta$ with a conformal scalar 
$f={(1+\sqrt{3})k}/({2t})$. 
In this case, $\vartheta$ is regular (as noted in \remref{rem:reg=f}) and  
the manifold $(M,\f,\xi,\eta,\bar{g})$ is almost $\eta$-Einstein with functions
\[
a=-\frac{(1+\sqrt{3})(2n-1)+2n}{2t^2},\qquad b=0,\qquad c=\frac{(1+\sqrt{3})(2n-1)}{2t^2}
\]
and it admits an almost $\eta$-Ricci soliton with potential $\vartheta$ and functions
\[
\lm=-\frac{(1+\sqrt{3})(kt-2n+1)-2n}{2t^2},\qquad \mu=0,\qquad \nu=\frac{(1+\sqrt{3})(kt-2n+1)}{2t^2}.
\]

\section*{Acknowledgements}
The author was supported by projects MU21-FMI-008 and FP21-FMI-002 of the Scientific Research Fund,
University of Plovdiv Paisii Hilendarski, Bulgaria.

%
\section*{Conflict of interest}

The author declares that he has no conflict of interest.


\end{document}